\documentclass[12pt,psamsfonts]{amsart}
\usepackage{fullpage,amssymb,amsfonts,amsmath}
\usepackage[all,cmtip]{xy}
\usepackage{enumerate}
\usepackage{mathrsfs}
\usepackage{comment}
\usepackage{hyperref} 
\usepackage[capitalise]{cleveref}
\usepackage{floatrow}
\usepackage{url}
\let\oldtocsection=\tocsection

\let\oldtocsubsection=\tocsubsection

\let\oldtocsubsubsection=\tocsubsubsection

\renewcommand{\tocsection}[2]{\hspace{0em}\oldtocsection{#1}{#2}}
\renewcommand{\tocsubsection}[2]{\hspace{1em}\oldtocsubsection{#1}{#2}}
\renewcommand{\tocsubsubsection}[2]{\hspace{2em}\oldtocsubsubsection{#1}{#2}}
\theoremstyle{plain} 

\newtheorem{thm}{Theorem}[section]
\newtheorem{cor}[thm]{Corollary}

\newtheorem{lem}[thm]{Lemma}

\newtheorem{quest*}{Question}
\newtheorem{prob*}{Problem}

\theoremstyle{definition}

\theoremstyle{remark}

\numberwithin{equation}{section}
\newcounter{notation}

\DeclareUrlCommand\DOI{}
\newcommand{\doi}[1]{\href{http://dx.doi.org/#1}{\DOI{doi:#1}}}



\crefname{figure}{Figure}{Figures}
\theoremstyle{plain}
\newtheorem*{thm*}{Theorem}
\crefname{thm}{Theorem}{Theorems}
\crefname{cor}{Corollary}{Corollarys}
\newtheorem*{cor*}{Corollary}
\crefname{cor*}{Corollary}{Corollarys}
\crefname{lem}{Lemma}{Lemmas}
\crefname{prop}{Proposition}{Propositions}
\crefname{conj}{Conjecture}{Conjectures}
\newtheorem*{conj*}{Conjecture}
\crefname{conj*}{Conjecture}{Conjectures}
\crefname{defn}{Definition}{Definitions}

\theoremstyle{remark}
\newtheorem*{rem*}{Remark}
\newtheorem*{rems*}{Remarks}

\def\addsymbol #1: #2#3{$#1$ \> \parbox{5.4in}{#2 \dotfill \pageref{#3}}\\} 
\def\addsymbolEND #1: #2#3{$#1$ \> \parbox{5.4in}{#2 \dotfill \pageref{#3}}}

\renewcommand{\bar}{\overline}

\newcommand{\C}{\mathbb{C}}
\newcommand{\cC}{C}

\newcommand{\kD}{\mathfrak{D}}

\newcommand{\kf}{\mathfrak{f}}

\newcommand{\Gal}{\mathrm{Gal}}

\newcommand{\sL}{\mathscr{L}}

\renewcommand{\Im}{\mathrm{Im}}

\newcommand{\N}{\mathrm{N}}

\newcommand{\kN}{\mathfrak{N}}

\newcommand{\kp}{\mathfrak{p}}

\newcommand{\kP}{\mathfrak{P}}

\newcommand{\cO}{\mathcal{O}}
\newcommand{\ord}{\mathrm{ord}\,}

\newcommand{\Q}{\mathbb{Q}}

\newcommand{\cR}{\mathcal{R}}
\newcommand{\R}{\mathbb{R}}

\renewcommand{\Re}{\mathrm{Re}}

\newcommand{\sumP}{\sideset{}{'}\sum}

\newcommand{\cZ}{\mathcal{Z}}

\title{Bounding the least prime ideal in the Chebotarev Density Theorem}

\author{Asif Zaman}
\thanks{The author was supported in part by an NSERC PGS-D scholarship.} 
\address{
Asif Zaman \\
Department of Mathematics \\ 
University of Toronto   \\ 
40 St. George Street, Room 6290 \\
Toronto \\
Canada \\
M5S 2E4}
\email{asif@math.toronto.edu}

\begin{document}

\begin{abstract} Let $K$ be a number field and suppose $L/K$ is a finite Galois extension. We establish a bound for the least prime ideal occurring in the Chebotarev Density Theorem. Namely, for every conjugacy class $\cC$ of $\mathrm{Gal}(L/K)$, there exists a prime ideal $\kp$ of $K$ unramified in $L$, for which its Artin symbol $\big[ \frac{L/K}{\kp} \big] = \cC$, for which its norm $\N^K_{\Q}\kp$ is a rational prime, and which satisfies 
\[
\N^K_{\Q} \kp \ll d_L^{40},
\]
where $d_L = |\mathrm{disc}(L/\Q)|$. All implicit constants are effective and absolute. 
\end{abstract}

\maketitle

\section{Introduction}
Let $K$ be a number field and $L/K$ be a finite Galois extension. For an unramified prime ideal $\kp$ of $K$, let $\big[ \frac{L/K}{\kp} \big]$ be its associated Artin symbol, which is a conjugacy class of $G := \Gal(L/K)$. For a given conjugacy class $\cC$ of $G$ and $X \geq 2$, 
define
\[
\pi_{\cC}(X) := \#\Big\{ \kp \text{ prime ideals of $K$ of degree 1} : \N^K_{\Q}\kp \leq X,  \kp \text{ unramified in $L$}, \Big[ \frac{L/K}{\kp} \Big] = \cC \Big\},
\]
where  $\N^K_{\Q}$ is the absolute norm of $K$. 
The Chebotarev Density Theorem \cite{Heilbronn, Chebotarev} states
\[
\pi_{\cC}(X)  \sim \frac{|\cC|}{|G|} \mathrm{Li}(X),
\]
where $\mathrm{Li}(X) = \int_2^X (\log t)^{-1} dt$, so infinitely many such prime ideals exist. One may then ask: when does such a prime ideal $\kp$ of least norm occur? 

Assuming the Generalized Riemann Hypothesis, Lagarias and Odlyzko \cite{LO} proved that
\[
\N^K_{\Q} \kp \ll (\log d_L)^2 (\log \log d_L)^4,
\]
where  $d_L = |\mathrm{disc}(K/\Q)|$ is the absolute discriminant of $L$. They also sketched a proof showing the $(\log\log d_L)^4$ factor could be removed entirely. Additionally assuming the Artin Conjecture, V.K. Murty \cite{Murty-K_LPC} showed that
\[
\N^K_{\Q} \kp \ll \frac{n_K^2}{|\cC|} \big( \log d_L  + [L:K] \log [L:K] \big)^2,
\]
where $n_K = [K:\Q]$ is the degree of $K/\Q$.

Unconditionally, Lagarias, Mongtomery and Odlyzko \cite{LMO} proved that
\begin{equation}
\N^K_{\Q} \kp \ll d_L^B
\label{LMOBound}
\end{equation}
for some effectively computable absolute constant $B > 0$. In \cite{KadNg}, Kadiri and Ng made reference to some explicit value of $B$ but the author has been unable to locate that preprint\footnote{\emph{Note added}: A preprint of this paper was posted on the arXiv in August 2015 (\texttt{arXiv/1508.00287}). Subsequently, in January 2016, the author was informed by Kadiri and Ng of their unpublished work \cite{KadNg_CDT} wherein they prove \cref{Theorem 1.1} in the case $K = \mathbb{Q}$.}.  The purpose of this paper is to show that $B = 40$ is admissible in \eqref{LMOBound}. 

\begin{thm} \label{Theorem 1.1}
Let $K$ be a number field and suppose $L/K$ is a finite Galois extension. For every conjugacy class $\cC$ of $\mathrm{Gal}(L/K)$, there exists a prime ideal $\kp$ of $K$ unramified in $L$, for which its Artin symbol $\big[ \frac{L/K}{\kp} \big] = \cC$, for which its norm $\N^K_{\Q}\kp$ is a rational prime, and which satisfies 
\[
\N^K_{\Q} \kp \ll d_L^{40},
\]
where $d_L = |\mathrm{disc}(L/\Q)|$. The implied constant is effective and absolute. 
\end{thm}
\begin{rem*} 
~
\begin{enumerate}[(i)]
	\item  	In several cases, one can reduce the exponent $B = 40$ by straightforward modifications. For example, one can take
\[
B = \begin{cases}
36.5 & \text{if $L$ has a tower of normal extensions with base $\Q$}, \\
24.1 & \text{if $n_L = o(\log d_L)$}, \\
7.5 & \text{if $\zeta_L(s)$ does not have a real zero $\beta_1 = 1 - \frac{\lambda_1}{\log d_L}$ satisfying $\lambda_1 = o(1)$,} \\
\end{cases}
\]
where $\zeta_L(s)$ is the Dedekind zeta function of $L$. See the remark at the end of \cref{sec:LPI} for details. 

	\item With a slight addition to our arguments, one can deduce a quantitative lower bound for $\pi_C(X)$. See \cite[Theorem 1.3.1]{Zaman_Thesis} for details. 
\end{enumerate}
\end{rem*}

The proof of \cref{Theorem 1.1} is motivated by the original arguments of \cite{LMO} which are naturally connected with Linnik's celebrated result \cite{Linnik1} on the least rational prime in an arithmetic progression. As such, we take advantage of powerful techniques found in Heath-Brown's work \cite{HBLinnik} on Linnik's constant. We also require explicit estimates related to the zeros of the Dedekind zeta function of $L$, denoted $\zeta_L(s)$. Recall
\begin{equation}
\zeta_L(s) = \sum_{\kN} (\N^L_{\Q} \kN)^{-s}
\label{def:DedekindZeta}
\end{equation}
for $s \in \C$ with $\Re\{s\} > 1$ and where the sum is over integral ideals $\kN$ of $L$. One key ingredient in our proof is an explicit zero-free region due to Kadiri \cite{Kadiri}. She showed that $\zeta_L(s)$ has at most one zero in the rectangle
\[
\Re\{s\} > 1 - \frac{0.0784}{\log d_L}, \qquad |\Im\{s\}| \leq 1.
\]
Furthermore, if such a zero $\beta_1$ exists, it is real and simple, and we refer to it as \emph{exceptional}. To handle this exceptional zero $\beta_1 = 1 - \frac{\lambda_1}{\log d_L}$, as  Linnik \cite{Linnik2} did for Dirichlet $L$-functions, we use explicit versions of Deuring-Heilbronn phenomenon for the Dedekind zeta function. We employ such a result due to Kadiri and Ng \cite{KadNg} when $\lambda_1 \gg 1$. To cover the remaining case when $\lambda_1 = o(1)$, which we refer to as a \emph{Siegel zero}, we establish another variant of Deuring-Heilbronn phenomenon. 

\begin{thm} \label{DH-AllZeros}
Suppose $\zeta_L(s)$ has a  real zero $\beta_1$ and let $\rho' = \beta'+i\gamma'$ be another zero of $\zeta_L(s)$ satisfying 
\begin{equation}
\tfrac{1}{2} \leq \beta' < 1 \quad \text{ and } \quad |\gamma'| \leq 1.
\label{DH-AllZeros_Range}
\end{equation}
Then,  for $d_L$ sufficiently large,
\[
\beta' \leq 1 - \frac{ \log\Big( \dfrac{c}{(1-\beta_1) \log d_L} \Big) }{ 35.8 \log d_L},
\]
where $c > 0 $ is an absolute effective constant.
\end{thm}
\begin{rems*} $ $
	\begin{enumerate}[(i)] 
	
	\item Kadiri and Ng \cite{KadNg} alternatively show that if
	\begin{equation}
	1 - \frac{\log\log d_L}{13.84 \log d_L} \leq \beta' < 1, \qquad |\gamma'| \leq 1,
	\label{KadNg_Range}
	\end{equation}
	and $d_L$ is sufficiently large then
	\[
	\beta' \leq 1 - \frac{ \log\Big( \dfrac{1}{(1-\beta_1) \log d_L} \Big) }{1.53 \log d_L}. 
	\]
	While the repulsion constant $1.53$ is much better than $35.8$ given by \cref{DH-AllZeros}, the permitted range of $\beta'$ in  \eqref{DH-AllZeros_Range} is much larger than that of \eqref{KadNg_Range} therefore allowing \cref{DH-AllZeros} to deal with Siegel zeros. 
	
	\item If $n_L = o(\log d_L)$ then $35.8$ can be replaced by $24.01$. By a classical theorem of Minkowski, recall $n_L = O(\log d_L)$ so such an assumption is often reasonable. 
	\end{enumerate} 
\end{rems*}

\cref{DH-AllZeros} gives a quantitative bound for \cite[Theorem 5.1]{LMO} and its proof is motivated by this non-explicit version. It involves a careful application of a modified Tur\'{a}n power sum inequality along with several explicit estimates concerning sums over zeros of $\zeta_L(s)$. Using similar arguments, we may establish a quantitative Deuring-Heilbronn phenomenon for only the real zeros of $\zeta_L(s)$ which is stronger than \cref{DH-AllZeros}. 
\begin{thm}  \label{DH-RealZeros} Suppose $\zeta_L(s)$ has a real zero $\beta_1$ and let $\beta'$ be another real zero of $\zeta_L(s)$ satisfying $0 < \beta' < 1$. Then,  for $d_L$ sufficiently large,
 \[
\beta' \leq 1 - \frac{ \log\Big( \dfrac{c}{(1-\beta_1) \log d_L} \Big) }{16.6 \log d_L},
\]
where $c > 0 $ is an absolute effective constant.
\end{thm}
\begin{rem*} Similar to remark (ii) following \cref{DH-AllZeros}, if $n_L = o(\log d_L)$ then $16.6$ can be replaced by $12.01$. 
\end{rem*}
Applying the above theorem to the zero $\beta' = 1-\beta_1$ of $\zeta_L(s)$ immediately yields the following corollary which will play a role in our proof of \cref{Theorem 1.1}. 

\begin{cor} \label{DH-RealZeros_Corollary} Suppose $\zeta_L(s)$ has a real zero $\beta_1$.  Then, for $d_L$ sufficiently large,
\[
1-\beta_1 \gg d_L^{-16.6},
\]
where the implied constant is absolute and effective. 
\end{cor}
\begin{rems*} \cref{DH-RealZeros_Corollary} makes explicit \cite[Corollary 5.2]{LMO} and so, as remarked therein, Stark \cite{Stark} gives a better bound for $1-\beta_1$ when $L$ has a tower of normal extensions with base $\Q$.
\end{rems*}

Finally, we describe the organization of the paper. \cref{sec:Prelim} provides the necessary preliminaries including background on the Dedekind zeta function, a power sum inequality, and some technical estimates.  \cref{sec:DH} contains work on the Deuring-Heilbronn phenomenon proving \cref{DH-AllZeros,DH-RealZeros}. \cref{sec:WeightedSum} prepares for the proof of \cref{Theorem 1.1} and \cref{sec:LPI} contains the concluding arguments divided into the relevant cases. 

\subsection*{Acknowledgements} I am very grateful to my advisor, Prof. John Friedlander, for his valuable suggestions and helpful conversations during our meetings, and for being extremely encouraging and supportive.
\section{Preliminaries}
\label{sec:Prelim}


\subsection{Dedekind zeta function}
\label{DedekindZeta}
The background material discussed here on the Dedekind zeta function can be found in \cite{LO,Heilbronn}. Consider a number field $L/\Q$ of degree $n_L = [L:\Q]$  with absolute discriminant $d_L= |\mathrm{disc}(L/\Q)|$ and ring of integers $\cO_L$. The \emph{Dedekind zeta function of $L$}, denoted $\zeta_L(s)$, can be given as a Dirichlet series by \eqref{def:DedekindZeta} or as an Euler product by 
\[
\zeta_L(s) = \prod_{\kP} \Big(1 - (\N^L_{\Q} \kP)^{-s} \Big)^{-1}
\] 
for $\Re\{s\} > 1$, where the product is over prime ideals $\kP$ of $L$. The \emph{completed Dedekind zeta function} $\xi_L(s)$ is given by
\begin{equation}
\xi_L(s) = s(s-1) d_L^{s/2} \gamma_L(s) \zeta_L(s),
\label{CompletedZeta}
\end{equation}
where $\gamma_L$ is the \emph{gamma factor of $L$} defined by
\begin{equation}
\gamma_L(s) =  \Big[ \pi^{-\tfrac{s}{2}} \Gamma\big(\frac{s}{2}\big)  \Big]^{r_1+r_2} \cdot \Big[ \pi^{-\tfrac{s+1}{2} } \Gamma\big( \frac{s+1}{2} \big)   \Big]^{r_2}.
\label{GammaFactor} 
\end{equation}
Here $r_1 = r_1(L)$ and $2r_2 = 2r_2(L)$ are the number of real and complex embeddings of $L$ respectively. It is well-known that $\xi_L(s)$ is entire and satisfies the functional equation
\begin{equation}
\xi_L(s) = \xi_L(1-s). 
\label{FunctionalEquation}
\end{equation}
We refer to its zeros as the \emph{non-trivial zeros $\rho$} of $\zeta_L(s)$, which are known to lie in the strip $0 < \Re\{s\} < 1$.  The \emph{trivial zeros $\omega$} of $\zeta_L(s)$ occur at certain non-positive integers arising from poles of the gamma factor of $L$; namely,
\begin{equation}
\mathop{\ord}_{s = \omega} \zeta_L(s) = 
\begin{cases}
r_1+r_2 - 1 & \text{if } \omega = 0, \\
r_2 &  \text{if } \omega = -1,-3,-5,\dots,\\
r_1+r_2 & \text{if } \omega = -2,-4,-6, \dots.
\end{cases}
\label{TrivialZeros}
\end{equation}
Using the functional equation and a Hadamard product for $\xi_L(s)$, one can deduce an explicit formula for the logarithmic derivative of $\zeta_L(s)$ given by the lemma below. 

\begin{lem} \label{ExplicitFormula} For any number field $L$ and $s \in \C$,  
\[
- \Re\Big\{ \frac{\zeta_L'}{\zeta_L}(s) \Big\} = \frac{1}{2} \log d_L + \Re\Big\{ \frac{1}{s-1} - \sum_{\rho } \frac{1}{s-\rho}  + \frac{1}{s} + \frac{\gamma_L'}{\gamma_L}(s) \Big\},
\]
where the sum is over all the non-trivial zeros $\rho$ of $\zeta_L(s)$. 
\end{lem}
\begin{proof} See \cite[Lemma 5.1]{LO} for example. 
\end{proof}

\subsection{Power Sum Inequality}
We record a power sum inequality and its proof from \cite[Theorem 4.2]{LMO} specialized to our intended application.  

\begin{lem}\cite[Lemma 4.1]{LMO}
\label{PowerSum-Kernel}
Define
\[
P(r, \theta) := \sum_{j=1}^J \Big(1 -\frac{j}{J+1}\Big) r^j \cos(j\theta). 
\]
Then
\begin{enumerate}[(i)]
	\item $P(r,\theta) \geq -\tfrac{1}{2}$ for $0 \leq r \leq 1$ and all $\theta$. 
	\item $P(1,0) = J/2$.
	\item $|P(r,\theta)| \leq \tfrac{3}{2}r$ for $0 \leq r \leq 1/3$. 
\end{enumerate}
\end{lem}

\begin{thm} \label{PowerSum} Let $\epsilon > 0$ and a sequence of complex numbers $\{z_n\}_n$ be given. Let $s_m = \sum_{n=1}^{\infty} z_n^m$ and suppose that $|z_n| \leq |z_1|$ for all $n \geq 1$. Define
\begin{equation}
M := \frac{1}{|z_1|}\sum_{n} |z_n|. 
\label{PowerSum-M}
\end{equation}
Then there exists $m_0$ with $1 \leq m_0 \leq (12+\epsilon) M$ such that
\[
\Re\{ s_{m_0} \} \geq \frac{\epsilon}{48+5\epsilon} |z_1|^{m_0}. 
\]
\end{thm}
\begin{proof} This is a simplified version of \cite[Theorem 4.2]{LMO}; our focus was to reduce their constant $24$ to $12+\epsilon$ by some minor modifications. We reiterate the proof here for clarity. Rescaling we may suppose $|z_1| = 1$. Write $z_n = r_n \exp(i \theta_n)$ so $r_n \in [0,1]$. 
 Then
\begin{equation*}
\begin{aligned}
S_J := & \sum_{j=1}^J \Big(1 -\frac{j}{J+1}\Big) \Re\{ s_j\} (1+\cos j \theta_1) \\
& \qquad = \sum_{n=1}^{\infty} \sum_{j=1}^J \Big(1 -\frac{j}{J+1}\Big) (\cos j \theta_n)(1+\cos j \theta_1) r_n^j \\
& \qquad = \sum_{n=1}^{\infty} \big\{ P(r_n, \theta_n) + \tfrac{1}{2}P(r_n, \theta_n-\theta_1) + \tfrac{1}{2} P(r_n, \theta_n+\theta_1) \big\}.
\end{aligned}
\end{equation*}
Using \cref{PowerSum-Kernel}, we estimate the contribution of each term. For $n=1$, we obtain a contribution $\geq \Big( \tfrac{J+1}{4} - r_1 \Big)$.  Terms $n > 1$ satisfying $r_n \geq 1/3$ contribute $\geq -1 \geq -3r_n$. Each of the remaining terms satisfying $r_n < 1/3$ are bounded using \cref{PowerSum-Kernel}(iii) and so contribute $\geq -3r_n$. Choosing $J = \lfloor (12+\epsilon) M \rfloor$, we deduce
\begin{equation}
S_J \geq \frac{J+1}{4} - 3M \geq  \frac{\epsilon M}{4} 
 \label{PowerSum-LB}
\end{equation}
as $J+1 \geq (12+\epsilon)M$. Now, suppose for a contradiction that $\Re\{ s_j \} < \tfrac{\epsilon}{48+5\epsilon}$ for all $1 \leq j \leq J$. Then, as $(1-\tfrac{j}{J+1})(1+\cos j\theta_1)$ is non-negative for all $1 \leq j \leq J$, 
\begin{equation*}
\begin{aligned}
S_J \leq \frac{\epsilon}{48+5\epsilon}  \sum_{j=1}^J \Big(1 -\frac{j}{J+1}\Big) (1+\cos j \theta_1)  < \frac{\epsilon}{48+5\epsilon}  \cdot 2 P(1,0) = \frac{\epsilon J}{48+5\epsilon} . 
\end{aligned}
\end{equation*}
Comparing with \eqref{PowerSum-LB} and noting $J \leq (12+\epsilon) M$, we obtain a contradiction. 
\end{proof}

\subsection{Technical Estimates} For the application of the power sum inequality, we will require some precise numerical estimates.
 
\begin{lem}
\label{ReGammaFactor} For $\alpha > 0$ and $t \geq 0$, 
\begin{align*}
\Re\Big\{\frac{\gamma_L'}{\gamma_L}(\alpha+1) + \frac{\gamma_L'}{\gamma_L}(\alpha+1\pm it) \Big\} 
& = G_1(\alpha; t) \cdot r_1 + G_2(\alpha; t) \cdot 2 r_2,
\end{align*}
where 
\begin{equation}
\begin{aligned}
G_1(\alpha; t) & := \frac{\Delta(\alpha+1,0) + \Delta(\alpha+1,t)}{2} - \log \pi,  \\
G_2(\alpha; t) & := \frac{\Delta(\alpha+1,0) + \Delta(\alpha+2,0) +  \Delta(\alpha+1,t) + \Delta(\alpha+2,t)}{4}  - \log \pi,
\end{aligned}
\label{ReGammaFactor_G}
\end{equation}
and $\Delta(x,y) = \Re\{ \frac{\Gamma'}{\Gamma}( \frac{x+iy}{2} ) \}$. 
\end{lem}
\begin{rem*} For fixed $\alpha > 0$ and $j=1$ or $2$, observe that $G_j(\alpha; t)$ is increasing as a function of $t \geq 0$ by \cite[Lemma 2]{Ahn-Kwon}. 
\end{rem*}
\begin{proof} Denote $\sigma = \alpha+1$. As $\Delta(x,y) = \Delta(x,-y)$, we may assume $t \geq 0$. From \eqref{GammaFactor}, it follows that
\begin{align*}
\Re\Big\{ \frac{\gamma_L'}{\gamma_L}(\sigma+it) \Big\} 
& = \frac{1}{2} \Big[ (r_1+r_2) \Delta(\sigma, t) + r_2 \Delta(\sigma+1,t)  - (r_1+2r_2) \log \pi  \Big] \\
& = \frac{1}{2} \Big[ r_1 ( \Delta(\sigma,t) - \log \pi) + 2r_2 \cdot  ( \tfrac{\Delta(\sigma,t) + \Delta(\sigma+1,t)}{2} - \log \pi )\Big] . 
\end{align*}
Using the same identity for $t=0$ gives the desired result. 
\end{proof}

 \begin{lem} \label{DH-ZeroSum} For $\alpha \geq 1$ and $t \in \R$, 
\begin{align*}
& \sum_{\omega} \Big( \frac{1}{|\alpha+1-\omega|^2} + \frac{1}{|\alpha+1+it-\omega|^2} \Big)  \\
& \qquad \leq \frac{1}{\alpha} \log d_L + \Big( \frac{G_1(\alpha; |t| )}{\alpha} + 2 W_1(\alpha) \Big) \cdot r_1 + \Big(\frac{G_2(\alpha; |t|)}{\alpha} +  W_2(\alpha) \Big) \cdot 2r_2  + \frac{2}{\alpha^2} + \frac{2}{\alpha+\alpha^2},  
\end{align*}
where the sum is over all zeros $\omega$ of $\zeta_L(s)$ including trivial ones, the functions $G_j(\alpha; |t|)$ are defined by \eqref{ReGammaFactor_G}, 
\[
W_1(\alpha) = \sum_{k = 0}^{\infty} \frac{1}{(\alpha+1+2k)^2}, \quad \text{and} \quad W_2(\alpha) = \sum_{k=0}^{\infty}  \frac{1}{(\alpha+1+k)^2}.  
\]
\end{lem}
\begin{proof} We estimate the trivial and non-trivial zeros separately. From \eqref{TrivialZeros}, notice
\begin{align*}
& \sum_{\omega \text{ trivial}} \frac{1}{|\alpha+1-\omega|^2}   \leq r_1 \sum_{k = 0}^{\infty} \frac{1}{(\alpha+1+2k)^2}  + r_2 \sum_{k=0}^{\infty}  \frac{1}{(\alpha+1+k)^2}. 
\end{align*}
Hence, 
\begin{equation}
\sum_{\omega \text{ trivial}} \Big( \frac{1}{|\alpha+1-\omega|^2} + \frac{1}{|\alpha+1+it-\omega|^2} \Big) \leq 
2 W_1(\alpha) \cdot r_1  + W_2(\alpha) \cdot 2r_2.
\label{MBound_Trivial}
\end{equation}
For the non-trivial zeros $\rho = \beta+i\gamma$, we combine the inequality
\[
0 \leq - \Re\Big\{ \frac{\zeta_L'}{\zeta_L}(\alpha+1) +  \frac{\zeta_L'}{\zeta_L}(\alpha+1+it) \Big\}
\]
with \cref{ExplicitFormula,ReGammaFactor} to deduce that
\begin{equation}
\begin{aligned}
0 & \leq \log d_L  + G_1(\alpha; |t|) \cdot r_1 + G_2(\alpha; |t|) \cdot 2r_2 + \Re\Big\{ \frac{1}{\alpha+it} + \frac{1}{\alpha+1+it} \Big\}  \\
& \qquad - \sum_{\rho} \Re\Big\{  \frac{1}{\alpha+1-\rho} + \frac{1}{\alpha+1+it-\rho}  \Big\} + \frac{1}{\alpha} + \frac{1}{\alpha+1} .
\end{aligned}
\label{BoundZeroSum}
\end{equation}
Observe, as $\beta \in (0,1)$, 
\[
\Re\Big\{ \frac{1}{\alpha+1+it-\rho} \Big\} = \frac{\alpha+1-\beta}{|\alpha+1+it-\rho|^2} \geq \frac{\alpha}{|\alpha+1+it-\rho|^2}
\]
and
\[
 \Re\Big\{ \frac{1}{\alpha+it} + \frac{1}{\alpha+1+it} \Big\}  \leq  \frac{1}{\alpha} + \frac{1}{\alpha+1}. 
\]
We rearrange \eqref{BoundZeroSum} and employ these observations to find that
\begin{equation}
\begin{aligned}
& \sum_{\rho} \Big( \frac{1}{|\alpha+1-\rho|^2} + \frac{1}{|\alpha+1+it-\rho|^2} \Big) \\
& \qquad\qquad  \leq \frac{1}{\alpha}\big( \log d_L + G_1(\alpha; |t|) \cdot r_1 + G_2(\alpha; |t|) \cdot 2r_2 \big) + \frac{2}{\alpha^2} + \frac{2}{\alpha+\alpha^2}. 
\end{aligned}
\label{MBound_Nontrivial}
\end{equation}
Combining with \eqref{MBound_Trivial} yields the desired bound. 
\end{proof}

\subsection{Choice of Weights} In the proof of \cref{Theorem 1.1}, we will need to select a suitable weight function so we describe our choice and its properties here. 

\begin{lem} \label{LPI-Weights} For real numbers $A,B > 0$ and positive integer $\ell \geq1$ satisfying $B > 2\ell A $, there exists a real-variable function $f(t) = f_{\ell}(t)$ such that:
\begin{enumerate}[(i)]
	\item $0 \leq f(t) \leq A^{-1}$ for all $t \in \R$.
	\item The support of $f$ is contained in $[B-2\ell A, B]$.
	\item Its Laplace transform $F(z) = F_{\ell}(z) = \int_{\R} f_{\ell}(t) e^{-zt} dt$ is given by
	\begin{equation}
	F(z) = e^{-(B- 2\ell A)z} \Big( \frac{1-e^{-Az}}{Az} \Big)^{2\ell}. 
	\label{WeightFunction-Laplace}
	\end{equation}
	\item Let $\sL \geq 1$ be arbitrary. Suppose $s = \sigma+ i t\in \C$ satisfies $\sigma < 1$ and $t \in \R$. Write $\sigma = 1 - \frac{x}{\sL}$ and $t = \frac{y}{\sL}$. If $0 \leq \alpha \leq 2\ell$ then  
	\[
	|F((1-s)\sL)| \leq e^{-(B-2\ell A) x} \Big(\frac{2}{A\sqrt{x^2+y^2}}\Big)^{\alpha} =  e^{-(B-2\ell A) (1-\sigma) \sL} \Big(\frac{2}{A|s-1|\sL}\Big)^{\alpha} .
	\]
	Furthermore,  
	\[
	|F((1-s)\sL)| \leq e^{-(B-2\ell A)x} \quad \text{and} \quad F(0) = 1. 
	\]
\end{enumerate}
\end{lem}
\begin{rem*} Heath-Brown \cite{HBLinnik} used the weight $f_{\ell}$ with $\ell = 1$ for his computation of Linnik's constant for the least rational prime in an arithmetic progression. 
\end{rem*}
\begin{proof}  
~
\begin{itemize}
	\item For parts (i)--(iii), let $\mathbf{1}_S(\, \cdot \,)$ be an indicator function for the set $S \subseteq \R$. For $j\geq1$, define 
\[
w_0(t) := \frac{1}{A} \mathbf{1}_{[-A/2,A/2]}(t),  \quad \text{and} \quad
w_j(t) := (w \ast w_{j-1})(t).
\]
Since $\int_{\R} w_0(t) dt = 1$, it is straightforward verify that $0 \leq w_{2\ell}(t) \leq A^{-1}$ and $w_{2\ell}(t)$ is supported in $[-\ell A, \ell A]$. Observe the Laplace transform $W(z)$ of $w_0$ is given by
\[
W(z) = \frac{e^{Az/2} - e^{-Az/2}}{A z} = e^{Az/2} \cdot \Big( \frac{1-e^{-Az}}{Az} \Big), 
\]
so the Laplace transform $W_{2\ell}(z)$ of $w_{2\ell}$ is given by
\[
W_{2\ell}(z) = \Big(\frac{e^{Az/2} - e^{-Az/2}}{A z} \Big)^{2\ell} = e^{\ell A z} \Big( \frac{1- e^{-Az}}{A z} \Big)^{2\ell}. 
\]
The desired properties for $f$ follow upon choosing $f(t) = w_{2\ell}(t- B+\ell A)$.
	\item For part (iv), we see by (iii)  that 
\begin{equation}
|F((1-s)\sL)| \leq  e^{-(B-2\ell A) x} \Big| \frac{1-e^{-A (x+iy)}}{A(x+iy)} \Big|^{2\ell}. 
\label{eqn:LPI-Weights_iv}
\end{equation}
To bound the above quantity, we observe that for $w = a+ib$ with $a > 0$ and $b \in \R$,
\[
\Big|\frac{1-e^{-w}}{w}\Big|^2 \leq \Big(\frac{1-e^{-a}}{a}\Big)^2 \leq 1.
\]
This observation can be checked in a straightforward manner (cf. \cref{CalcBound}). It follows that
\begin{equation*}
\begin{aligned}
\Big| \frac{1-e^{-A (x+iy)}}{A(x+iy)} \Big|^{2\ell}
= \Big| \frac{1-e^{-A (x+iy)}}{A(x+iy)} \Big|^{ \alpha } \cdot \Big| \frac{1-e^{-A (x+iy)}}{A(x+iy)} \Big|^{2\ell-\alpha}
\leq \Big( \frac{2}{A\sqrt{x^2+y^2}}\Big)^{\alpha}. 
\end{aligned}
\end{equation*}
In the last step, we noted $|1 -e^{-A(x+iy)}| \leq 2$ since $x > 0$ by assumption. Combining this with \eqref{eqn:LPI-Weights_iv} yields the desired bound. The additional estimate for $|F((1-s)\sL)|$ is the case when $\alpha = 0$. One can  verify $F(0) = 1$ by straightforward calculus arguments. 
\end{itemize}
\end{proof}

\begin{lem} \label{CalcBound} For $z = x+iy$ with $x > 0$ and $y \in \R$, 
\[
\Big| \frac{1-e^{-z}}{z}\Big|^{2} \leq \Big(  \frac{1-e^{-x}}{x} \Big)^{2}.
\]
\end{lem}
\begin{proof} We need only consider $y \geq 0$ by conjugate symmetry. Define
\[
\Phi_x(y) := \Big| \frac{1-e^{-z}}{z}\Big|^{2} =  \frac{1+e^{-2x} - 2e^{-x} \cos y }{x^2+y^2} \qquad \text{for $y \geq 0$},
\]
which is a non-negative smooth function of $y$. Since $\Phi_x(y) \rightarrow 0$ as $y \rightarrow \infty$, we may choose $y_0 \geq 0$ such that $\Phi_x(y)$ has a global maximum at $y = y_0$. Suppose, for a contradiction, that 
\begin{equation}
\Phi_x(y_0) > \Big( \frac{1-e^{-x}}{x} \Big)^2.
\label{CalcBound-1}
\end{equation}
By calculus, one can show $(1-e^{-x})/x \geq e^{-x/2}$ for $x > 0$. With this observation, notice that
\begin{align*}
\Phi_x'(y_0) 
& = \frac{2 e^{-x} \cdot \sin y_0}{x^2+y_0^2} - \frac{ 2 \Phi_x(y_0) \cdot y_0 }{x^2+y_0^2} \\
& < \frac{2 e^{-x} \cdot \sin y_0}{x^2+y_0^2} - \frac{ 2 \big( \frac{1-e^{-x}}{x} \big)^2 \cdot y_0 }{x^2+y_0^2} & \text{by \eqref{CalcBound-1}} \\
& \leq \frac{2 e^{-x} \cdot \sin y_0}{x^2+y_0^2} - \frac{ 2 e^{-x} \cdot y_0 }{x^2+y_0^2}  \\
& \leq 0,
\end{align*}
since $\sin y \leq y$ for $y \geq 0$.  On the other hand, $\Phi_x(y)$ has a global max at $y = y_0$ implying $\Phi_x'(y_0) = 0$, a contradiction. 
\end{proof}

\section{Deuring-Heilbronn Phenomenon}
\label{sec:DH}
In this section, we prove \cref{DH-MainTheorem,DH-RealZeros}. Notice \cref{DH-AllZeros} is contained in \cref{DH-MainTheorem} below.

\begin{thm} \label{DH-MainTheorem}
Let $T \geq 1$ be fixed. Suppose $\zeta_L(s)$ has a real zero $\beta_1$ and let $\rho' = \beta'+i\gamma'$ be another zero of $\zeta_L(s)$ satisfying 
\[
\tfrac{1}{2} \leq \beta' < 1 \quad \text{ and } \quad |\gamma'| \leq T.
\]
Then  for $d_L$ sufficiently large
\[
\beta' \leq 1 - \frac{ \log\Big( \dfrac{c}{(1-\beta_1) \log d_L} \Big) }{C \log d_L},
\]
where $c = c(T) > 0 $ and $C = C(T) > 0$ are absolute effective constants. In particular, one may take $T$ and $C = C(T)$ according to the table below. 
\[
\begin{array}{l|c|c|c|c|c|c|c|c|c|c|c|}
T & 1 & 3.5 & 8.7 & 22 & 54 & 134 & 332 & 825 & 2048 & 5089 & 12646 \\
\hline
C & 35.8 & 37.0 & 39.3 & 42.5 & 46.1 & 50.0 & 53.8 & 57.6 & 61.4 & 65.2 & 69.0 \\
\end{array}
\]
\end{thm}
\begin{rems*} $ $
\begin{enumerate}[(i)]
	\item This result for general $T \geq 1$ follows from \cite[Theorem 5.1]{LMO} but our primary concern is verifying the table of values for $T$ and $C$.   The choices of $T$ in the given table are obviously not special; one can compute $C$ for any fixed $T$ by a simple modification to our argument below.  We made these selections primarily for their application in the proof of \cref{Theorem 1.1} in \cref{sec:LPI}. 
	\item If $n_L = o(\log d_L)$ then one can take $C = 24.01$ for any fixed $T$. 
\end{enumerate}
\end{rems*}

\subsection{Proof of \cref{DH-MainTheorem}} 

Let $m$ be a positive integer and $\alpha \geq 1$. 
From \cite[Equation (5.4)]{LMO} with $s = \alpha+1 + i\gamma'$, it follows that
\begin{equation}
\Re\Big\{ \sum_{n=1}^{\infty}  z_n^m \Big\} \leq \frac{1}{\alpha^m} - \frac{1}{(\alpha+1-\beta_1)^{2m}} + \Re\Big\{ \frac{1}{(\alpha+i\gamma')^{2m}} - \frac{1}{(\alpha +i\gamma' + 1-\beta_1)^{2m}} \Big\}, 
\label{DH-AllZeros_PowerSumPrep}
\end{equation}
where $z_n = z_n(\gamma')$ satisfies $|z_1| \geq |z_2| \geq \dots$ and runs over the multiset 
\begin{equation}
\{  (\alpha+1-\omega)^{-2}, (\alpha+1+i\gamma' - \omega)^{-2} :  \omega \neq \beta_1 \text{ is any zero of $\zeta_L(s)$} \}. 
\label{DH-AllZeros_z_n}
\end{equation}
Note that the multiset includes trivial zeros of $\zeta_L(s)$. With this choice, we have that
\begin{equation}
(\alpha+1-\beta')^{-2} \leq |z_1| \leq \alpha^{-2}. 
\label{DH-zProperty}
\end{equation}
Since
\[
\Big| \frac{1}{(\alpha+it)^{2m}} - \frac{1}{(\alpha + it + 1-\beta_1)^{2m}} \Big| \leq \alpha^{-2m} \Big| 1 - \frac{1}{(1 + \frac{1-\beta_1}{\alpha+it})^{2m} } \Big|   \ll \alpha^{-2m-1} m(1-\beta_1),
\]
equation \eqref{DH-AllZeros_PowerSumPrep} implies
\begin{equation}
\Re\Big\{ \sum_{n=1}^{\infty}  z_n^m \Big\} \ll \alpha^{-2m-1} m(1-\beta_1). 
\label{DH-PowerSum_RHS}
\end{equation}
On the other hand, by \cref{PowerSum}, for $\epsilon > 0$, there exists some $m_0 = m_0(\epsilon)$ with $1 \leq m_0 \leq (12+\epsilon)M$ such that
\[
\Re\Big\{ \sum_{n=1}^{\infty} z_n^{m_0} \Big\} \geq \tfrac{\epsilon}{50}  |z_1|^{m_0} \geq \tfrac{\epsilon}{50} (\alpha+1-\beta')^{-2m_0} \geq \tfrac{\epsilon}{50} \alpha^{-2m_0} \exp(-\tfrac{2m_0}{\alpha}(1-\beta') ), 
\]
where $M= |z_1|^{-1} \sum_{n=1}^{\infty} |z_n|$ according to our parameters $z_n = z_n(\gamma')$ in \eqref{DH-AllZeros_z_n}. Comparing with \eqref{DH-PowerSum_RHS} for $m = m_0$, it follows that
\begin{equation}
\exp(-(24+2\epsilon)\tfrac{M}{\alpha} (1-\beta') ) \ll_{\epsilon} \tfrac{M}{\alpha} (1-\beta_1). 
\label{DH-Penultimate}
\end{equation}
Therefore, it  suffices to bound $M/\alpha$ and optimize over $\alpha \geq 1$. By \cref{DH-ZeroSum} and \eqref{DH-zProperty}, notice that
\begin{equation}
\label{DH-Penultimate_2}
\begin{aligned}
\frac{M}{\alpha}  & \leq \frac{(\alpha+1-\beta')^{2}}{\alpha} \cdot \Big\{\frac{1}{\alpha} \log d_L + \Big( \frac{G_1(\alpha; |\gamma'|)}{\alpha} + 2 W_1(\alpha) \Big) \cdot r_1 \\
& \qquad\qquad + \Big(\frac{G_2(\alpha; |\gamma'|)}{\alpha} +  W_2(\alpha) \Big) \cdot 2r_2  + \frac{2}{\alpha^2} + \frac{2}{\alpha+\alpha^2}  \Big\}
\end{aligned}
\end{equation}
for $\alpha \geq 1$. To simplify the above, we note $1-\beta' \leq 1/2$ by assumption and $G_j(\alpha; |\gamma'|) \leq G_j(\alpha; T)$ for $j=1,2$ by the remark following \cref{ReGammaFactor}.  Also in \eqref{DH-Penultimate_2}, if a coefficient of $r_1$ or $r_2$ is positive, we employ an estimate of Odlyzko \cite{Odlyzko1977} which implies
\begin{equation}
(\log 60) \cdot r_1 + (\log 22) \cdot 2r_2 \leq \log d_L
\label{Odlyzko}
\end{equation}
for $d_L$ sufficiently large. With these observations, it follows that
\begin{equation*}
\begin{aligned}
\frac{M}{\alpha} & \leq  \frac{(\alpha+1/2)^{2}}{\alpha}\bigg[  \Big( \frac{1}{\alpha} + \max\Big\{  \frac{G_1(\alpha; T) + 2 \alpha W_1(\alpha) }{\alpha \log 60}, \frac{G_2(\alpha; T)+\alpha W_2(\alpha)}{\alpha \log 22}  , 0 \Big\} \Big) \log d_L \\
& \qquad\qquad\qquad\qquad  +  \frac{2}{\alpha^2} + \frac{2}{\alpha+\alpha^2} \bigg].
\end{aligned}
\label{DH-Penultimate_3}
\end{equation*}
Seeking to minimize the coefficient of $\log d_L$, after some numerical calculations, we choose $\alpha = \alpha(T)$ according to the following table:
\[
\begin{array}{l|c|c|c|c|c|c|c|c|c|c|c|}
T & 1 & 3.5 & 8.7 & 22 & 54 & 134 & 332 & 825 & 2048 & 5089 & 12646 \\
\hline
\alpha & 3.07 & 4.06 & 5.68 & 7.73 & 9.43 & 10.7 & 11.7 & 12.7 & 13.7 & 14.7 & 15.7 \\
\end{array}
\]
To complete the proof for $T=1$, say, the corresponding choice of $\alpha = 3.07$ implies
\[
\Big. \frac{M}{\alpha}  \leq 1.4883 \log d_L 
\]
for $d_L$ sufficiently large. Substituting this bound into \eqref{DH-Penultimate} and fixing $\epsilon > 0$ sufficiently small yields the desired result since $24 \times 1.4883 < 35.8$. The other cases follow similarly. \qed
\begin{rem*} 
~
\begin{itemize}
	\item 	To clarify remark (ii) following \cref{DH-MainTheorem}, notice that if $n_L = o(\log d_L)$ then the coefficients of $r_1$ and $r_2$ in \eqref{DH-Penultimate_2} can be made arbitrary small for $d_L$ sufficiently large depending on $\alpha \geq 1$. Fixing $\alpha$ sufficiently large (depending on $T$) gives 
\[
M/\alpha \leq 1.0001 \log d_L
\]
for $d_L$ sufficiently large. As $24 \times 1.0001 < 24.01$ the remark follows.  
	\item All computations were performed using \textsc{Maple}. Relevant code can be obtained either on the author's personal webpage or by email request. 
\end{itemize}
\end{rem*}
\subsection{Proof of \cref{DH-RealZeros}} The proof is very similar to the above proof for \cref{DH-MainTheorem} with a few differences which we outline here. Recall $\beta'$ is now a real zero of $\zeta_L(s)$ distinct from $\beta_1$ (counting with multiplicity). As before, let $m$ be a positive integer and $\alpha \geq 1$. From \cite[Equation (5.4)]{LMO} with $s = \alpha+1$ instead, it follows that
\begin{equation}
\Re\Big\{ \sum_{n=1}^{\infty}  z_n^m \Big\} \leq \frac{1}{\alpha^m} - \frac{1}{(\alpha+1-\beta_1)^{2m}}, 
\label{DH-RealZeros_PowerSumPrep_1}
\end{equation}
where $z_n$ satisfies $|z_1| \geq |z_2| \geq \dots$ and runs over the multiset 
\[
\{  (\alpha+1-\omega)^{-2}  :  \omega \neq \beta_1 \text{ is any zero of $\zeta_L(s)$} \}. 
\]
If $\omega$ is a trivial zero (and hence a non-positive integer by \eqref{TrivialZeros}) then $(\alpha+1-\omega)^{-2} \geq 0$. Thus, for any $z_n$ in \eqref{DH-RealZeros_PowerSumPrep_1} corresponding to a trivial zero, we have $z_n^m \geq 0$ so we may discard such $z_n$. It follows that
\begin{equation}
\Re\Big\{ \sum_{n=1}^{\infty}  \tilde{z}_n^m \Big\} \leq \frac{1}{\alpha^m} - \frac{1}{(\alpha+1-\beta_1)^{2m}}, 
\label{DH-RealZeros_PowerSumPrep}
\end{equation}
where $\tilde{z}_n$ satisfies $|\tilde{z}_1| \geq |\tilde{z}_2| \geq \dots$ and runs over the new (smaller) multiset 
\begin{equation}
\{  (\alpha+1-\rho)^{-2}  :  \rho \neq \beta_1 \text{ is any non-trivial zero of $\zeta_L(s)$} \}. 
\label{DH-RealZeros_z_n}
\end{equation}
For this new choice of $\tilde{z}_n$, the analogue of \eqref{DH-zProperty} still holds for $\tilde{z}_1$ and we argue similarly to deduce \eqref{DH-Penultimate} holds for the new $\tilde{M} = |\tilde{z}_1|^{-1} \sum_n |\tilde{z}_n|$. Thus, by the proof of \cref{DH-ZeroSum} (namely by \eqref{MBound_Nontrivial} with $t=0$), we deduce that
\begin{equation}
\label{DH2-Penultimate_2}
\begin{aligned}
\frac{\tilde{M}}{\alpha}  & \leq \frac{(\alpha+1-\beta')^{2}}{2\alpha} \cdot \Big\{\frac{1}{\alpha} \log d_L + \frac{G_1(\alpha; 0)}{\alpha} \cdot r_1 +  \frac{G_2(\alpha; 0)}{\alpha} \cdot 2r_2  + \frac{2}{\alpha^2} + \frac{2}{\alpha+\alpha^2}  \Big\}
\end{aligned}
\end{equation}
for $\alpha \geq 1$. Comparing with \eqref{DH-Penultimate_2}, notice the additional factor of $2$ in the denominator and the lack of $W_1(\alpha)$ and $W_2(\alpha)$ terms. Continuing to argue analogously, we simplify the above by noting $1-\beta' < 1$ and applying Odlyzko's bound \eqref{Odlyzko} to conclude
\begin{align*}
\frac{M}{\alpha} & \leq  \frac{(\alpha+1)^{2}}{2\alpha}\bigg[  \Big( \frac{1}{\alpha} + \max\Big\{  \frac{G_1(\alpha; 0)}{\alpha \log 60}, \frac{G_2(\alpha; 0)}{\alpha \log 22}, 0\Big\} \Big) \log d_L \\
& \qquad\qquad\qquad\qquad  +  \frac{2}{\alpha^2} + \frac{2}{\alpha+\alpha^2} \bigg]
\end{align*}
for $d_L$ sufficiently large. Selecting $\alpha = 5.8$ gives
\[
\frac{\tilde{M}}{\alpha} \leq 0.6882 \log d_L
\]
for $d_L$ sufficiently large. As $24 \times 0.6882 < 16.6$, we similarly conclude the desired result.  
\qed
\section{Weighted Sum of Prime Ideals}
\label{sec:WeightedSum}
\subsection{Setup} For the remainder of the paper, denote
\[
\sL = \log d_L.
\]
Suppose the integer $\ell \geq 2$ and real numbers $A,B > 0$ satisfy $B - 2\ell A > 0$. Select the weight function $f$ from \cref{LPI-Weights} according to these parameters. Assume $2 \leq B \leq 100$ henceforth, while $\ell$ and $A$ remain arbitrary.  

Recall $K$ is a number field with ring of integers $\cO_K$ and $L/K$ is a finite Galois extension. Let $\cC$ be a conjugacy class of $G := \Gal(L/K)$. Define
\begin{equation}
S := \sum_{\substack{ \kp \text{ unramified in $L$} \\ \N\kp = p \text{ rational prime}}} \frac{\log \N\kp}{\N\kp} f\Big( \frac{\log \N\kp}{\sL} \Big) \cdot \mathbf{1}\Big\{ \Big[ \frac{L/K}{\kp} \Big] = \cC \Big\}, 
\label{def:S}
\end{equation}
where $\N = \N^K_{\Q}$ is the absolute norm of $K$, $\mathbf{1}\{ \cdot \}$ is an indicator function, and $\big[ \frac{L/K}{\kp} \big]$ is the Artin symbol of $\kp$. To prove \cref{Theorem 1.1}, we claim it suffices to show $S > 0$ for $d_L$ sufficiently large and a suitable choice of parameters $A,B$ and $\ell$; in particular, we must take $B \leq 40$. By our choice of $f$,  it would follow that there exists an unramified prime ideal $\kp$ of degree 1 with $\big[ \frac{L/K}{\kp} \big] = \cC$ satisfying $\N^K_{\Q}\kp \leq  d_L^B$ for $d_L$ sufficiently large. For all values of $d_L$ which are not sufficiently large, the result follows from \eqref{LMOBound} (that is, \cite[Theorem 1.1]{LMO}).  This proves the claim. 

Now, we wish to transform $S$ into a contour integral by using the logarithmic derivatives of certain Artin $L$-functions. One is naturally led to consider the contour
\begin{equation}
I := \frac{1}{2\pi i} \int_{2-i\infty}^{2+i\infty} \Psi_{\cC}(s) F( (1-s) \sL) ds 
\label{def:J}
\end{equation}
with
\begin{equation}
\Psi_{\cC}(s) := - \frac{|\cC|}{|G|} \sum_{\psi} \bar{\psi}(g) \frac{L'}{L}(s,\psi, L/K),
\label{Psi_C-Artin}
\end{equation}
where $g \in \cC$, the sum runs over irreducible characters $\psi$ of $\Gal(L/K)$, and $L(s,\psi,L/K)$ is the Artin $L$-function attached to $\psi$. By orthogonality of characters (see \cite[Section 3]{Heilbronn}), observe that
\begin{equation}
\Psi_{\cC}(s) = \sum_{\kp \subseteq \cO_K} \sum_{m=1}^{\infty}\frac{\log \N\kp}{(\N\kp^m)^{s} } \cdot  \Theta_C(\kp^m) \qquad \text{for $\Re\{s\} > 1$, }
\label{Psi_C-DirichletSeries}
\end{equation}
where, for prime ideals $\kp \subseteq \cO_K$ unramified in $L$,
\begin{equation}
\Theta_C(\kp^m) = \begin{cases} 1 & \text{if $\big[ \frac{L/K}{\kp} \big]^m \in \cC$}, \\
0 & \text{else},
\end{cases}
\label{LPI-Theta}
\end{equation} 
and $0 \leq \Theta_C(\kp^m) \leq 1$ for prime ideals $\kp \subseteq \cO_K$ ramified in $L$. Comparing \eqref{def:J} and \eqref{Psi_C-DirichletSeries}, it follows by Mellin inversion that
\begin{equation}
I = \sL^{-1} \sum_{\kp \subseteq \cO_K} \sum_{m=1}^{\infty}\frac{\log \N\kp}{\N\kp^m}  f\Big( \frac{\log \N\kp^m}{\sL} \Big)  \cdot  \Theta_C(\kp^m).
\label{J-MellinInversion}
\end{equation}

Comparing \eqref{J-MellinInversion} and \eqref{def:S}, it is apparent that the integral $I$ and quantity $\sL^{-1} S$ should be equal up to a neglible contribution from: (i) ramified prime ideals, (ii) prime ideals whose norm is not a rational prime, and (iii) prime ideal powers. In the following lemma, we prove exactly this by showing that the collective contribution of (i), (ii), and (iii) in \eqref{J-MellinInversion} is bounded by $O(A^{-1} \sL^2 e^{-(B-2\ell A)\sL/2})$. 

\begin{lem} 
\label{SequalsJ}
In the above notation, 
\[
\sL^{-1} S = \frac{1}{2\pi i} \int_{2-i\infty}^{2+i\infty} \Psi_{\cC}(s) F( (1-s) \sL) ds  + O( A^{-1} \sL^2 e^{-(B-2 \ell A)\sL/2}).
\]
\end{lem}
\begin{proof} Denote $Q_1 = e^{(B-2\ell A)\sL}$ and $Q_2 = e^{B\sL}$.  
\subsubsection*{Ramified prime ideals} Since the product of ramified prime ideals $\kp \subseteq \cO_K$ divides the different $\kD_{L/K}$, it follows that 
\[
\sum_{\substack{ \kp \subseteq \cO_K \\ \text{ramified in $L$}}} \log \N\kp \leq \log d_L = \sL.
\]
Therefore, by \cref{LPI-Weights},
\begin{equation*}
\begin{aligned}
&  \sum_{\substack{ \kp \subseteq \cO_K \\ \text{ramified in $L$}}} \sum_{m=1}^{\infty}\frac{\log \N\kp}{\N\kp^m}  f\Big( \frac{\log \N\kp^m}{\sL} \Big)  \cdot  \Theta_C(\kp^m)  \\
& \qquad \ll A^{-1} \sum_{\substack{ \kp \subseteq \cO_K \\ \text{ramified in $L$}}} \log \N\kp \sum_{\substack{ m \geq 1 \\  \N\kp^m > Q_1} } \frac{1}{\N\kp^m}  \\
& \qquad \ll A^{-1} \sum_{\substack{ \kp \subseteq \cO_K \\ \text{ramified in $L$} \\ \N\kp > Q_1}} \frac{ \log \N\kp}{\N\kp} \\
& \qquad \ll A^{-1} \sL e^{-(B-2\ell A)\sL}.
 \end{aligned}
\end{equation*}
\subsubsection*{Prime ideals with norm not equal to a rational prime}  For a given integer $q$, there are at most $n_K$ prime ideals $\kp \subseteq \cO_K$ satisfying $\N\kp = q$.  Thus, by \cref{LPI-Weights},
\begin{equation*}
\begin{aligned}
& \sum_{p \text{ prime}} \sum_{k \geq 2} \sum_{\substack{ \kp \subseteq \cO_K \\ \N\kp = p^k }} \frac{\log \N\kp}{\N\kp} f\Big( \frac{\log \N\kp}{\sL} \Big)  \cdot  \Theta_C(\kp) \\
& \qquad \ll A^{-1} n_K \sL \sum_{p \text{ prime}} \sum_{\substack{ k \geq 2  \\ Q_1< p^k < Q_2}}  \frac{ 1}{p^k}  \\
& \qquad \ll A^{-1} n_K \sL Q_1^{-1/2} \\
& \qquad \ll A^{-1} \sL^2 e^{-(B-2\ell A)\sL/2}.
\end{aligned}
\end{equation*}
Note in the last step we used the fact that $n_K \leq n_L \ll \sL$ by a theorem of Minkowski. 
\subsubsection*{Prime ideal powers}
Arguing similar to the previous case, one may again see that

\begin{equation*}
\begin{aligned}
& \sum_{p \text{ prime}} \sum_{\substack{ \kp \subseteq \cO_K \\ \N\kp = p }}  \sum_{m \geq 2} \frac{\log \N\kp}{\N\kp^m} f\Big( \frac{\log \N\kp^m}{\sL} \Big)  \cdot  \Theta_C(\kp^m) \ll A^{-1} \sL^2 e^{-(B-2\ell A)\sL/2}.
\end{aligned}
\end{equation*}
The desired result follows after comparing \eqref{def:S}, \eqref{def:J} and \eqref{J-MellinInversion} with the three estimates above. 
\end{proof}
\subsection{Deuring's reduction} Equipped with \cref{SequalsJ}, the natural next step is to move the contour to the left of $\Re\{s\} = 1$ but this poses a difficulty. Artin $L$-functions are not yet in general known to have meromorphic continuation in the left halfplane $\Re\{s\} \leq 1$. It is therefore not immediately clear that $\Psi_{\cC}(s)$ is defined in this region.  Thus, we employ a reduction due to Deuring \cite{Deuring} as argued in Lagarias-Montgomery-Odlyzko \cite[Section 3]{LMO} whose argument we repeat here for the sake of clarity. 

For $g \in \cC$, define the cyclic subgroup $H = \langle g \rangle$ of $\Gal(L/K)$ and let $E$ be the fixed field of $H$.  Then by \cite[Lemma 4]{Heilbronn}, 
\begin{equation}
\Psi_{\cC}(s) = - \frac{|\cC|}{|G|} \sum_{\chi} \bar{\chi}(g) \frac{L'}{L}(s,\chi, L/E),
\label{Psi_C-Artin_Deuring}
\end{equation}
where the sum runs over irreducible characters $\chi$ of $H$. These characters are necessarily $1$-dimensional since $H$ is abelian. By class field theory, the Artin $L$-function $L(s,\chi, L/E)$ is actually a certain Hecke $L$-function $L(s,\chi,E)$ since $L/E$ is abelian. Further, $\chi$ is a primitive Hecke character satisfying
\[
\chi(\kP) = \chi\Big( \Big[ \frac{L/E}{\kP} \Big]\Big)
\]
for all prime ideals $\kP \subseteq \cO_E$ unramified in $L$. Therefore, \eqref{Psi_C-Artin_Deuring} becomes
\begin{equation}
\Psi_{\cC}(s) = - \frac{|\cC|}{|G|} \sum_{\chi} \bar{\chi}(g) \frac{L'}{L}(s,\chi, E),
\label{Psi_C-Hecke}
\end{equation}
where $\chi$ are certain primitive Hecke characters of $E$. 
Note that, from \cite{Heilbronn} for example, 
\begin{equation}
\zeta_L(s) = \prod_{\chi} L(s,\chi, L/E)
\label{Zeta_AbelianFactorization}
\end{equation}
and the conductor-discriminant formula states
\begin{equation}
\log d_L = \sum_{\chi} \log(d_E \N^E_{\Q} \kf_{\chi}),
\label{Conductor-Discriminant}
\end{equation}
where $\kf_{\chi} \subseteq \cO_E$ is the conductor of $\chi$. 

\subsection{A sum over low-lying zeros} In light of \eqref{Psi_C-Hecke}, we are now in a position to use the analytic properties of Hecke $L$-functions and shift the contour in \cref{SequalsJ}. We will reduce the analysis to a careful consideration of contribution coming from zeros  $\rho = \beta+i\gamma$  of $\zeta_L(s)$ which are ``low-lying". 

\begin{lem} \label{S-LowLyingZeros} Let $T^{\star} \geq 1$ be fixed. In the above notation, 
\begin{equation}
\begin{aligned}
\Big| \tfrac{|G|}{|\cC|}\sL^{-1} S -  F(0) \Big| &  \leq  \sum_{\substack{ \rho  \\  |\gamma| < T^{\star} } }  |F((1-\rho)\sL)|  + O\Big( \sL\Big(\frac{2}{A T^{\star} \sL} \Big)^{2\ell}  + \frac{\sL^2}{A}  e^{-(B-2 \ell A) \sL/2}   \Big)  \\
& + O\Big(   \sL \Big( \frac{1}{A \sL} \Big)^{2\ell}   e^{-(B-2 \ell A) \sL} + \sL \Big( \frac{2}{A \sL} \Big)^{2\ell} e^{-3(B-2\ell A)\sL /2} \Big), 
\end{aligned}
\label{S-LowLyingZeros-EQ} 
\end{equation}
where the sum is over non-trivial zeros $\rho = \beta+i\gamma$ of $\zeta_L(s)$. 
\end{lem}
\begin{proof} Consider the contour in \cref{SequalsJ}. Using \eqref{Psi_C-Hecke}, we shift the line of integration to $\Re\{s\} = -\tfrac{1}{2}$. From \eqref{Zeta_AbelianFactorization}, this  picks up exactly the non-trivial zeros of $\zeta_L(s)$, its simple pole at $s=1$, and its trivial zero at $s=0$ of order $r_1+r_2-1$. For $\Re\{s\} = -1/2$,  we have by \cref{LPI-Weights} that
\begin{equation}
F((1-s)\sL) \ll e^{-3(B-2\ell A) \sL/2}\cdot \Big(\frac{2}{A \sL |s|}\Big)^{2\ell}
\label{S-LowLyingZeros_0}
\end{equation}
and, from \cite[Lemma 6.2]{LO} and \eqref{Conductor-Discriminant}, 
\begin{align*}
\sum_{\chi} |\frac{L'}{L}(s,\chi,E)|
&  \ll \sum_{\chi} \big\{ \log(d_E \N^E_{\Q}\kf_{\chi}) + n_E \log(|s|+2)  \big\}  \\
&  \ll \sL + [L:E] \cdot n_E \log(|s|+2)   \\
&  \ll \sL + n_L \log(|s|+2). 
\end{align*}
It follows that
\[
\frac{1}{2\pi i} \int_{-1/2-i\infty}^{-1/2+i\infty} \Psi_{\cC}(s) F( (1-s) \sL) ds \ll \sL \Big( \frac{2}{A \sL} \Big)^{2\ell} e^{-3(B-2\ell A)\sL/2}
\]
as $n_L \ll \sL$. For the zero at $s=0$ of $\Psi_{\cC}(s)$, we may bound its contribution using \eqref{WeightFunction-Laplace} to deduce that
\[
(r_1+r_2-1) F(\sL) \ll \sL \Big( \frac{1}{A \sL} \Big)^{2\ell}  e^{-(B-2 \ell A) \sL}, 
\]
since $r_1+2r_2 = n_L \ll \sL$. These observations and  \cref{SequalsJ} therefore yield
\begin{equation}
\begin{aligned}
\Big| \tfrac{|G|}{|\cC|}\sL^{-1} S -  F(0)   \Big| & \leq  \sum_{\rho} |F((1-\rho)\sL)|  + O\Big(  \frac{\sL^2}{A}e^{-(B-2 \ell A) \sL/2}  + \sL \Big( \frac{1}{A \sL} \Big)^{2\ell} e^{-(B-2 \ell A) \sL} \Big) \\
& \qquad\qquad + O\Big( \sL \Big( \frac{2}{A \sL} \Big)^{2\ell} e^{-3(B-2\ell A)\sL /2}\Big),
\end{aligned}
\label{S-LowLyingZeros_1}
\end{equation}
where the sum is over all non-trivial zeros $\rho = \beta + i\gamma$ of $\zeta_L(s)$. By \cite[Lemma 2.1]{LMO} and \cref{LPI-Weights}, we have that
\[
\sum_{k = 0}^{\infty}  \sum_{\substack{ \rho \\ T^{\star} + k \leq |\gamma| < T^{\star} + k+1} } |F((1-\rho)\sL)|  \ll \Big(\frac{2}{A\sL} \Big)^{2\ell} \sum_{k = 0}^{\infty} \frac{\sL + n_L \log(T^{\star}+k)}{(T^{\star}+k)^{2\ell}} \ll  \sL \Big(\frac{2}{A T^{\star} \sL} \Big)^{2\ell},
\]
as $n_L \ll \sL$, $T^{\star}$ is fixed, and $\ell \geq 2$.  The result follows from \eqref{S-LowLyingZeros_1}  and the above estimate. 
\end{proof}

For the sum over low-lying zeros in \cref{S-LowLyingZeros}, we bound zeros far away from the line $\Re\{s\} = 1$ using \cref{S-LowLyingZeros_PartialSummation} below. In the non-exceptional case, this could have been done in a fairly simple manner but when a Siegel zero exists, we will need to partition the zeros according to their height. This will amount to applying a coarse version of partial summation, allowing us to exploit the Deuring-Heilbronn phenomenon more efficiently.   
\begin{lem} \label{S-LowLyingZeros_PartialSummation} Let $J \geq 1$ be given and $T^{\star} \geq 1$ be fixed. Suppose 
\[
1 \leq R_1 \leq R_2 \leq \dots \leq R_J \leq \sL, \quad 0 = T_0 < T_1 \leq T_2 \leq \dots \leq T_J = T^{\star}.
\]
Then
\begin{equation}
\begin{aligned}
\sum_{\substack{ \rho  \\  |\gamma| < T^{\star} } }  |F((1-\rho)\sL)| & =  \sumP_{\rho} |F((1-\rho)\sL)| + O\Big( \min\Big\{ \Big(\frac{2}{A}\Big)^{2\ell}, \sL \Big\} e^{-(B-2\ell A)R_1}  \Big) \\
& \qquad + \sum_{j=2}^J O\Big(  \sL\Big(\frac{2}{A T_{j-1} \sL} \Big)^{2\ell} e^{-(B-2\ell A) R_j}   \Big), \\
\end{aligned}
\label{S-LowLyingZeros_PartialSummation_EQ} 
\end{equation}
where the marked sum $\sum'$ indicates a restriction to zeros $\rho = \beta+i\gamma$ of $\zeta_L(s)$ satisfying  
\[
\beta > 1 - \frac{R_j}{\sL}, \qquad T_{j-1} \leq |\gamma| < T_j \quad \text{for some $1 \leq j \leq J$}.
\]
If $J = 1$ then the secondary error term in \eqref{S-LowLyingZeros_PartialSummation_EQ}  vanishes. 
\end{lem}
\begin{rem*} To prove \cref{Theorem 1.1}, we will apply the above lemma with $J = 10$ when a Siegel zero exists. One could use higher values of $J$ or a more refined version of \cref{S-LowLyingZeros_PartialSummation} to obtain some improvement on the final result. 
\end{rem*}
\begin{proof} Recall $\ell \geq 2$ for our choice of weight $f$. Let $1 \leq j \leq J$ be arbitrary. Define the multiset
\begin{align*}
\cZ_j & := \{ \rho : \zeta_L(\rho) = 0, \,\,\, \beta \leq 1-\frac{R_j}{\sL}, \,\, T_{j-1} \leq |\gamma| < T_j \}
\end{align*}
and denote $S_j := \sum_{\rho \in \cZ_j} |F((1-\rho)\sL)|$. Since
\[
\sum_{\substack{ \rho  \\  |\gamma| < T^{\star} } }  |F((1-\rho)\sL)| =  \sumP_{\rho} |F((1-\rho)\sL)| +  \sum_{j=1}^J S_j,
\]
it suffices to show
\begin{align*}
S_1 & \ll \min\Big\{ \Big(\frac{2}{A}\Big)^{2\ell}, \sL \Big\}  e^{-(B-2\ell A) R_1} \\
\text{and} \quad S_j & \ll \sL\Big(\frac{2}{A T_{j-1} \sL} \Big)^{2\ell} e^{-(B-2\ell A) R_j},  \quad \text{for $2 \leq j \leq J$}. 
\end{align*}
Assume $2 \leq j \leq J$. As $T_j \leq T^{\star}$ and $T^{\star}$ is fixed, it follows $\# \cZ_j \ll \sL$ by \cite[Lemma 2.1]{LMO}.  Hence, by \cref{LPI-Weights} and the definition of $\cZ_j$, 
\[
S_j \ll e^{-(B-2\ell A) R_j} \sum_{\rho \in \cZ_j} \Big(\frac{2}{A|\gamma| \sL} \Big)^{2\ell} \ll \sL \Big(\frac{2}{AT_{j-1} \sL} \Big)^{2\ell} e^{-(B-2\ell A) R_j},
\]
as desired.  It remains to consider $S_1$. On one hand, we similarly have $\# \cZ_1 \ll \sL$ by \cite[Lemma 2.1]{LMO}. Thus, by \cref{LPI-Weights} and the definition of $S_1$, 
\begin{equation}
S_1 \ll \sL e^{-(B-2\ell A) R_1}.  
\label{S-LowLyingZeros_S1naive}
\end{equation}
On the other hand, we may give an alternate bound for $S_1$. For integers $1 \leq m,n \leq \sL$, consider the rectangles
\[
\cR_{m,n} := \Big\{s = \sigma+it \in \C : 1-\frac{m+1}{\sL} \leq \sigma \leq 1-\frac{m}{\sL}, \quad \frac{n-1}{\sL} \leq |t| \leq  \frac{n}{\sL} \Big\}.
\] 
We bound the contribution of zeros $\rho$ lying in $\cR_{m,n}$ when $m \geq R_1$. If a zero $\rho \in \cR_{m,n}$ then 
\[
|F((1-\rho)\sL)| \ll e^{-(B-2 \ell A)m} \Big(\frac{2}{A\sqrt{m^2+(n-1)^2}}\Big)^{2\ell},
\]
by \cref{LPI-Weights} with $\alpha = 2\ell$. Further, by \cite[Lemma 2.2]{LMO}, 
\[
\#\{ \rho \in \cR_{m,n} : \zeta_L(\rho) = 0  \}   \ll \sqrt{(m+1)^2+n^2} \ll \sqrt{m^2 + (n-1)^2}. 
\]
The latter estimate follows since $m,n \geq 1$. Adding up these contributions and using the conjugate symmetry of zeros, we find that
\begin{align*}
S_1 \ll \sum_{\substack{ m \geq R_1 \\ n \geq 1} }\,\,\sum_{\substack{ \rho \in \cR_{m,n} \\ \zeta_L(\rho) =0 } } |F((1-\rho)\sL)|
&  \ll \Big(\frac{2}{A}\Big)^{2\ell} \sum_{\substack{ m \geq R_1 \\ n \geq 1} } e^{-(B-2\ell A)m} \big(\sqrt{m^2+(n-1)^2}\big)^{-2\ell+1} \\
&  \ll \Big(\frac{2}{A}\Big)^{2\ell}   e^{-(B-2\ell A) R_1},
\end{align*}
since $\ell \geq 2$. Taking the minimum of the above and \eqref{S-LowLyingZeros_S1naive} gives the desired bound for $S_1$. 
\end{proof}

If a Siegel zero exists then we shall choose the parameters in \cref{S-LowLyingZeros_PartialSummation} so that the restricted sum over zeros is actually empty. Otherwise, if a Siegel zero does not exist then \cref{S-LowLyingZeros_PartialSummation} will be applied with $J=1$ and $T_1 = T^{\star} = 1$ so we must handle the remaining restricted sum over zeros in the final arguments.
\begin{lem} \label{S-LowLyingZerosBound} Let $\eta > 0$ and $R \geq 1$  be arbitrary.  For $A > 0$ and $\ell \geq 1$, define
\[
\tilde{F}_{\ell}(z) := \Big( \frac{1-e^{-Az}}{Az}\Big)^{2\ell}.
\]
Suppose $\zeta_L(s)$ is non-zero in the region
\[
\Re\{s\} \geq 1 - \frac{\lambda}{\sL}, \qquad |\Im\{s\}| \leq 1,
\]
for some $0 < \lambda \leq 10$.  Then, provided $d_L$ is sufficiently large depending on $\eta, R,$ and $A$, 
\begin{equation}
\sumP_{\rho} |\tilde{F}_{\ell}((1-\rho)\sL)| \leq \Big( \frac{1-e^{-A\lambda}}{A\lambda} \Big)^{2(\ell-1)} \cdot \Big\{ \phi \Big( \frac{1-e^{-2A\lambda}}{A^2\lambda} \Big) + \frac{2A\lambda-1+e^{-2A\lambda}}{2A^2\lambda^2} + \eta \Big\},
\label{S-LowLyingZerosBound_EQ}
\end{equation}
where  $\phi = \tfrac{1}{2}(1-\tfrac{1}{\sqrt{5}})$ and the marked sum $\sum'$ indicates a restriction to zeros $\rho = \beta+i\gamma$ of $\zeta_L(s)$ satisfying
\[
\beta \geq 1 - \frac{R}{\sL}, \qquad |\gamma| \leq 1.
\]
In particular, as $\lambda \rightarrow 0$, the bound in \eqref{S-LowLyingZerosBound_EQ} becomes $\frac{2\phi}{A} + 1 + \eta$. 
\end{lem}
\begin{proof} This result is motivated by \cite[Lemma 13.3]{HBLinnik}. Define 
\[
h(t) := 
\begin{cases}
A^{-2} \cdot \sinh\big( (A-t)\lambda \big) & \text{if } 0 \leq t \leq A, \\
0 & \text{if } t \geq A, 
\end{cases}
\]
so 
\[
H(z) = \int_0^{\infty} e^{-zt} h(t)dt = \frac{1}{2A^2} \Big\{ \frac{e^{A\lambda}}{\lambda+z} + \frac{e^{-A\lambda}}{\lambda-z} - \frac{2\lambda e^{-Az}}{\lambda^2-z^2} \Big\}.
\]
As per the argument in \cite[Lemma 13.3]{HBLinnik},  
\begin{equation}
|\tilde{F}_1(\lambda+z)|  \leq \frac{2 e^{-A\lambda}}{\lambda} \cdot \Re\{ H(z) \}
 \label{S-LowLyingZerosBound_2} 
\end{equation}
for $\Re\{z\} \geq 0$. Combining the above with \cref{CalcBound} and noting $(1-e^{-x})/x$ is decreasing for $x > 0$, it follows that
\[
|\tilde{F}_{\ell} (\lambda+z)|  \leq \Big( \frac{1-e^{-A\lambda}}{A\lambda}\Big)^{2(\ell-1)} \cdot \frac{2 e^{-A\lambda}}{\lambda}  \cdot \Re\{ H(z) \}
\]
for $\Re\{z\} \geq 0$.   Setting $\sigma = 1 - \frac{\lambda}{\sL} \in \R$, this implies
\[
\sumP_{\rho} |\tilde{F}_{\ell}((1-\rho)\sL)| \leq  \Big( \frac{1-e^{-A\lambda}}{A\lambda}\Big)^{2(\ell-1)} \cdot \frac{2 e^{-A\lambda}}{\lambda}   \sumP_{\rho} \Re\{ H((\sigma-\rho)\sL) \},
\]
so it suffices to bound the sum on the RHS. Since $h$ and $H$ satisfy Conditions 1 and 2 of \cite{KadNg}, we apply \cite[Theorem 3]{KadNg} to bound the sum $\sum'$ on the RHS yielding
\begin{equation*}
\begin{aligned}
\sumP_{\rho} \Re\{ H((\sigma-\rho)\sL) \} & \leq h(0)(\phi+\eta) + H((\sigma-1)\sL) - \sL^{-1} \sum_{\kN \subseteq \cO_L} \frac{\Lambda_L(\kN)}{(\N^L_{\Q}\kN)^{\sigma}} h\Big( \frac{\log \N^L_{\Q}\kN}{\sL} \Big) \\
& \leq h(0)(\phi+\eta) + H((\sigma-1)\sL),
\end{aligned}
\end{equation*}
for $d_L$ sufficiently large depending on $\eta, R,$ and $A$.  Using the definitions of $h$ and $H$ and rescaling $\eta$ appropriately, we obtain the desired result. 
\end{proof}
\section{Proof of \cref{Theorem 1.1}}
\label{sec:LPI}
Let $\mathcal{Z}$ be the multiset consisting of zeros of  $\zeta_L(s)$ in the rectangle
\[
0 < \Re\{s\} < 1, \quad |\Im\{s\}| \leq 1. 
\]
Choose $\rho_1 \in \cZ$ such that $\Re\{\rho_1\} = \beta_1 = 1-\frac{\lambda_1}{\sL} \in (0,1)$ is maximal.  If $\lambda_1 < 0.0784$ then $\rho_1$ is exceptional; that is, $\rho_1$ is a simple real zero of $\zeta_L(s)$ as shown by Kadiri \cite{Kadiri}. We divide our arguments according to this exceptional case. Recall that our goal is  to show the quantity $S$, defined by \eqref{def:S}, is strictly positive for $d_L$ sufficiently large and $B \leq 40$. 
\subsection{Non-Exceptional Case $(\lambda_1 \geq 0.0784)$} 
Choose 
\[
\ell = 2, \quad B = 7.41, \quad \text{and} \quad A = 1.5
\]
to give a corresponding $f$ and its Laplace transform $F$ defined by \cref{LPI-Weights}. Observe that $B-2\ell A = 1.41$ for the above choices. 

Let $\epsilon > 0$. Apply \cref{S-LowLyingZeros} with $T^{\star} = 1$. Then employ \cref{S-LowLyingZeros_PartialSummation} with $J = 1, T_1=T^{\star}=1$ and $R_1= R =  R(\epsilon)$ sufficiently large so that
\[
\frac{|G|}{|\cC|} \sL^{-1} S \geq 1 - \sumP_{\rho} |F((1-\rho)\sL)| - \epsilon
\]
for $d_L$ sufficiently large depending on $\epsilon$. Here the restricted sum is over zeros $\rho = \beta+i\gamma$ satisfying
\[
\beta > 1 - \frac{R}{\sL}, \qquad |\gamma| < 1. 
\]
 It suffices to prove the sum over zeros $\rho$ is $<1-\epsilon/2$ for fixed sufficiently small $\epsilon$. Observe by the definition of $\tilde{F}_2$ in \cref{S-LowLyingZeros} and our choice of  $\rho_1$ that
\[
\sumP_{\rho} |F((1-\rho)\sL)|  = \sumP_{\rho} e^{-1.41 \lambda} |\tilde{F}_2((1-\rho)\sL)|  \leq e^{-1.41 \lambda_1} \sumP_{\rho} |\tilde{F}_2((1-\rho)\sL)|. 
\]
Since $\lambda_1 \geq 0.0784$, we may bound the remaining sum using \cref{S-LowLyingZerosBound} with $\lambda = 0.0784$. Hence, the above is
\[
\leq e^{-1.41 \lambda_1} \times 1.1166 \leq e^{-1.41 \times 0.0784} \times 1.1166 =  0.9997\dots < 1,
\]
as required. 
\subsection{Exceptional Case $(\lambda_1 < 0.0784)$}
For this subsection, let $0 < \eta < 0.0784$ be an absolute parameter which will be specified later.  
\subsubsection{$\lambda_1$ small $(0.0784 > \lambda_1 \geq \eta)$}
Again, choose the weight function $f$ from \cref{LPI-Weights} with
\[
\ell = 2, \quad B = 2.63, \quad \text{and} \quad A=0.1
\]
so $B - 2\ell A = 2.23$. 
The argument is similar to the previous case but we take special care of the real zero $\beta_1$. By the same choices as the non-exceptional case, we deduce that
\begin{equation}
\frac{|G|}{|\cC|} \sL^{-1} S \geq 1 - |F((1-\beta_1)\sL)| - \sumP_{\rho \neq \beta_1} |F((1-\rho)\sL)|  - \epsilon
\label{T1-SemiExceptional}
\end{equation}
for $d_L$ sufficiently large depending on $\epsilon$. Observe that, since $\rho_1$ is real and $(1-e^{-t})/t \leq 1$ for $t > 0$,  
\[
|F((1-\rho_1)\sL)| = e^{-2.23 \lambda_1} \Big( \frac{1-e^{-0.1\lambda_1}}{0.1 \lambda_1} \Big)^4 \leq e^{-2.23 \lambda_1}. 
\]
Now, select another zero $\rho' \in \cZ$ of $\zeta_L(s)$ such that $\rho' \neq \rho_1$ (counting with multiplicity in $\cZ$) and $\Re\{\rho'\} = \beta' = 1-\frac{\lambda'}{\sL}$ is maximal. In the exceptional case, $\rho_1$ is a simple real zero so $\rho'$ is indeed genuinely distinct from $\rho_1$.  By our choice of $\rho'$, \cref{LPI-Weights}, and a subsequent application of \cref{S-LowLyingZerosBound} with $\lambda = 0$, we gave that
\[
\sumP_{\rho \neq \rho_1} |F((1-\rho)\sL)| \leq e^{-2.23 \lambda'}  \sumP_{\rho \neq \rho_1} |\tilde{F}_2((1-\rho)\sL)| \leq e^{-2.23 \lambda'} \times 6.5279. 
\]
As $\lambda_1 \geq \eta$, it follows that $\lambda' \geq 0.6546 \log \lambda_1^{-1}$ from \cite[Theorem 4]{KadNg} for $d_L$ sufficiently large depending on $\eta$. Hence, the above is
\[
\leq 6.5279 \times \lambda_1^{2.23 \times 0.6546} \leq 6.5279 \times \lambda_1^{1.4597}. 
\]
Thus, \eqref{T1-SemiExceptional} implies
\begin{align*}
\frac{|G|}{|\cC|} \sL^{-1} S & 
\geq 1-e^{-2.23\lambda_1} - 6.5279 \times \lambda_1^{1.4597} - \epsilon\\
& \geq  \big(2.23- 6.5279 \times \lambda_1^{0.4597} - 2.4865\lambda_1 \big) \lambda_1 - \epsilon,
\end{align*}
since $1-e^{-t} \geq t - t^2/2$ for $t > 0$. The quantity in the brackets is clearly decreasing with $\lambda_1$ so, since $\lambda_1 < 0.0784$, we conclude that the above is
\begin{align*}
& \geq   \big(2.23- 6.5279 \times 0.0784^{0.4597} - 2.4865 \times 0.0784 \big)\lambda_1  - \epsilon \\
& \geq   0.0097\lambda_1  - \epsilon. 
\end{align*}
As $\lambda_1 \geq \eta > 0$ by assumption, the result follows after taking $\epsilon = 10^{-6} \eta$. 

\subsubsection{$\lambda_1$ very small $(\sL^{-200} \leq \lambda_1 < \eta)$}
\label{subsec:VeryExceptional}
Choose the weight function $f$ from \cref{LPI-Weights} with
\[
\ell = 101, \quad B = 36.5, \quad \text{and} \quad A= \tfrac{1}{404},
\]
so $B - 2\ell A = 36$. Applying \cref{S-LowLyingZeros} with $T^{\star} =  1$, it follows that
\[
 \frac{|G|}{|C|} \sL^{-1} S \geq 1 - |F((1-\beta_1)\sL)| - \sum_{\substack{ \rho \neq \beta_1 \\  |\gamma| < 1} }  |F((1-\rho)\sL)|  + O(\sL^{-201}).
\]
Similar to the previous subcase, we have that $|F((1-\beta_1)\sL)| \leq e^{-36 \lambda_1}$. For the remaining sum over zeros, we apply \cref{S-LowLyingZeros_PartialSummation} with $J=1, T^{\star} = T_1 = 1$, and $R_1 = \frac{1}{35.8} \log(c_1/\lambda_1)$ with $c_1 > 0$ absolute and sufficiently small. As $\lambda_1 \geq \sL^{-200}$, we may assume without loss that $R_1 < \tfrac{1}{4}\sL$ for $\sL$ sufficiently large\footnote{This implies that the zero $1-\beta_1$ is already discarded in the error term arising from \cref{S-LowLyingZeros_PartialSummation}. This minor point will be relevant when $\lambda_1$ is extremely small.}. Therefore, 
\begin{equation}
\begin{aligned}
\frac{|G|}{|C|} \sL^{-1} S
&  \geq 1 - e^{-36\lambda_1} - \sumP_{\rho \neq \beta_1} |F((1-\rho)\sL)| +  O\big(\sL^{-201} + \lambda_1^{36/35.8} \big),
\label{T1-Exceptional-VerySmall}
\end{aligned}
\end{equation}
where the sum $\sum'$ is defined as per \cref{S-LowLyingZeros_PartialSummation}. By our choice of parameters $T_1$ and $R_1$,  it follows from \cref{DH-MainTheorem} that the restricted sum over zeros in \eqref{T1-Exceptional-VerySmall} is actually empty. As $1-e^{-t} \geq t - t^2/2$ for $t > 0$, we conclude that
\[
 \frac{|G|}{|C|} \sL^{-1} S \geq 36 \lambda_1  + O(\sL^{-201} + \lambda_1^{36/35.8}). 
\]
Since $\sL^{-200} \leq \lambda_1 < \eta$ by assumption and $\eta$ is sufficiently small, we conclude that the RHS is $\gg \lambda_1$ after fixing $\eta$.  

\subsubsection{$\lambda_1$ extremely small $(\lambda_1 < \sL^{-200})$}
\label{subsec:ExtremelyExceptional}
Choose the weight function $f$ from \cref{LPI-Weights} with
\[
\ell = \lceil  1.1 \sL \rceil , \quad B = 39.5, \quad \text{and} \quad A= \frac{0.9}{\sL},
\]
so $B - 2\ell A > 37.5$ for $d_L$ sufficiently large. Applying \cref{S-LowLyingZeros} with $T^{\star} =  12646$, it follows that

\begin{equation}
\begin{aligned}
\Big| \tfrac{|G|}{|\cC|}\sL^{-1} S -  F(0) \Big| 
&  \leq  \sum_{\substack{ \rho  \\  |\gamma| < 12646 } }  |F((1-\rho)\sL)|  + O\Big( \sL e^{2.2 \log\big( \tfrac{2}{0.9 \times 12646} \big) \sL} +   \sL^3 e^{-37.5 \sL/2} \Big)  \\
& \qquad  + O\Big(\sL e^{-37.5 \sL + 2.2 \log\big( \tfrac{1}{0.9} \big) \sL }   + \sL  e^{-\tfrac{3}{2} \times 37.5 \sL + 2.2 \log\big( \tfrac{2}{0.9} \big) \sL } \Big)  \\
& \leq  \sum_{\substack{ \rho  \\  |\gamma| < 12646 } }  |F((1-\rho)\sL)|  + O( e^{-18 \sL} ).
\label{T1-Exceptional}
\end{aligned}
\end{equation}
For the remaining sum, we use \cref{S-LowLyingZeros_PartialSummation} with $J=10$ selecting $T_j$ and $R_j = \frac{\log(c_j/\lambda_1)}{C_j}$ according to the table below. Note $c_j = c(T_j) > 0$ is the absolute constant in \cref{DH-MainTheorem}. 
\[
\begin{array}{l|c|c|c|c|c|c|c|c|c|c|}
j & 1 & 2 & 3 & 4 & 5 & 6 & 7 & 8 & 9 & 10 \\
\hline
T_j & 3.5 & 8.7 & 22 & 54 & 134 & 332 & 825 & 2048 & 5089 & 12646 \\
\hline
C_j  & 37.0 & 39.3 & 42.5 & 46.1 & 50.0 & 53.8 & 57.6 & 61.4 & 65.2 & 69.0 \\
\end{array}
\]
Therefore,
\begin{equation}
\begin{aligned}
\frac{|G|}{|\cC|} \sL^{-1} S
&  \geq 1 - |F((1-\beta_1)\sL)| - \sumP_{\rho \neq \beta_1, 1-\beta_1} |F((1-\rho)\sL)|  - |F(\beta_1 \sL)| + O(e^{-18\sL}) \\
& \qquad  +  O\big( \sL \lambda_1^{37.5/37.0} \big) +  \sum_{j=2}^{10} O\Big(  \sL e^{2.2\log\big( \tfrac{2}{0.9 T_{j-1}} \big) \sL}   \lambda_1^{37.5/C_j}   \Big),
\label{T1-Exceptional-2}
\end{aligned}
\end{equation}
where the sum $\sum'$ is defined as per \cref{S-LowLyingZeros_PartialSummation}. 
Since the zeros of $\zeta_L(s)$ are permuted under the map $\rho \mapsto 1-\rho$, it follows from \cref{DH-MainTheorem} and our choice of parameters $T_j$ and $C_j$ that the restricted sum over zeros in \eqref{T1-Exceptional-2} is actually empty. 
For the zeros $1-\beta_1$ and $\beta_1$, notice
\begin{align*}
|F((1-\beta_1)\sL)|  \leq e^{-37.5 \lambda_1} \quad \text{ and } \quad F(\beta_1\sL) \leq e^{-37.5(\sL-\lambda_1)} = O( e^{-37.5 \sL} )
\end{align*}
as $\lambda_1 < 0.0784$. 
Moreover, as $\lambda_1 < \sL^{-200}$ and $\tfrac{37.5}{37.0} > 1.01$, we observe that
\[
\sL \cdot \lambda_1^{37.5/37.0} \ll \lambda_1^{-1/200} \cdot \lambda_1^{1.01} \ll \lambda_1^{1.005}. 
\]
To bound the sum over error terms in \eqref{T1-Exceptional-2}, notice $\lambda_1 \gg \sL e^{-16.6 \sL}$  by \cref{DH-RealZeros_Corollary},  which implies
 \[
 \sL e^{2.2\log\big( \tfrac{2}{0.9 T_j} \big) \sL}   \lambda_1^{37.5/C_j}  \ll  \lambda_1 \cdot \sL^2 e^{ 2.2\log\big( \tfrac{2}{0.9 T_{j-1}} \big) \sL + 16.6 (1- 37.5/C_j) \sL}. 
 \]
 Substituting the prescribed values for $C_j$ and $T_{j-1}$, the above is $\ll \lambda_1 e^{-0.2 \sL}$ for all $2 \leq j \leq 10$. Incorporating all of these observations into \eqref{T1-Exceptional-2} yields
\begin{align*}
\frac{|G|}{|\cC|} \sL^{-1} S 
& \geq 1 - e^{-37.5 \lambda_1} + O\Big( \lambda_1^{1.005} + \lambda_1 e^{-0.2 \sL} + e^{-18 \sL} \Big) \\
& \geq   37.5 \lambda_1   + O\Big( \lambda_1^{1.005} + \lambda_1 e^{-0.2 \sL} + e^{-18 \sL} \Big), 
\end{align*}
since $1-e^{-t} \geq t - t^2/2$ for $t > 0$. Noting  $\lambda_1 \gg \sL e^{-16.6 \sL}$ by \cref{DH-RealZeros_Corollary}, we finally conclude that the RHS is positive for $d_L$ sufficiently large and $\lambda_1 < \sL^{-200}$. 

\begin{rem*} We outline the minor modifications required to justify the remark following \cref{Theorem 1.1}.  
\begin{itemize}
	\item If there is a sequence of fields $\Q = L_0 \subseteq L_1 \subseteq \cdots \subseteq L_r = L$ such that $L_{j}$ is normal over $L_{j-1}$ for $1 \leq j \leq r$ then, by \cite[Lemmas 10, 11]{Stark}, it follows that $\lambda_1 \gg \sL e^{-0.5 \sL}$. For \cref{subsec:ExtremelyExceptional}, one may therefore select
	\[
	\ell = \lceil  0.05 \sL \rceil , \quad B = 36.4 , \quad \text{and} \quad A= \frac{3.53}{\sL},
	\]
	and apply  \cref{S-LowLyingZeros} with $T^{\star} = 149$. Afterwards, employ \cref{S-LowLyingZeros_PartialSummation}  with $T_j$ and $R_j = \frac{\log(c_j/\lambda_1)}{C_j}$ chosen according to the table below.
	\[
	\begin{array}{l|c|c|c|}
j & 1 & 2 & 3 \\
\hline
T_j & 1 & 12.2 & 149 \\
\hline
C_j  & 35.8& 40.3 & 50.4 \\
\end{array}
	\]
	and follow the same arguments. This requires additional instances of \cref{DH-MainTheorem} with $T = 12.2$ and $149$ yielding $C(T) = 40.3$ amd $50.4$ respectively. 
	\item If $n_L = o(\log d_L)$ then, by the remark following \cref{DH-RealZeros}, we have that $\lambda_1 \gg \sL e^{-12.01 \sL}$. 		Moreover, by remark (ii) following \cref{DH-MainTheorem}, one can use 
	\[
	J = 1, \quad  T_1 = T^{\star} = e^{64}, \quad  \text{and} \quad R_1  = \frac{\log(c/\lambda_1)}{24.01}
	\]
	in the application of \cref{S-LowLyingZeros,S-LowLyingZeros_PartialSummation}. One may then modify \cref{subsec:VeryExceptional} to consider $\sL^{-1000} \leq \lambda_1 < \eta$ and take
	\[
	\ell = 1000, \quad B = 24.1, \quad A = 1/10^6. 
	\]
	Similarly, one may modify \cref{subsec:ExtremelyExceptional} to consider $\lambda_1 < \sL^{-1000}$ and take
	\[
	\ell = \lceil  0.1 \sL \rceil , \quad B = 24.1, \quad \text{and} \quad A= \frac{0.2}{\sL}. 
	\]
	Following the same arguments yields the claimed result. 
	\item If $\zeta_L(s)$ does not have a Siegel zero then $\lambda_1 \gg 1$ so \cref{subsec:VeryExceptional,subsec:ExtremelyExceptional} are unnecessary. 
\end{itemize}
\end{rem*}
\begin{rem*}
For \cref{subsec:ExtremelyExceptional}, the selection of parameters $A,B, \ell, T_1$ and $T_2$ was primarily based on numerical experimentation.
\end{rem*}

\bibliographystyle{alpha}
\bibliography{biblio}

\end{document}